\documentclass[12pt]{amsart}%
\usepackage{amsmath}
\usepackage{graphicx}
\usepackage{amsfonts}
\usepackage{amssymb}%
\providecommand{\U}[1]{\protect\rule{.1in}{.1in}}
\newtheorem{theorem}{Theorem}[section]
\newtheorem{lemma}[theorem]{Lemma}

\theoremstyle{definition}

\newtheorem{example}[theorem]{Example}
\theoremstyle{remark}

\numberwithin{equation}{section}

\begin{document}
\title[A generalization of the Banach-Steinhaus theorem]{A generalization of the Banach-Steinhaus theorem for finite part limits}
\author{Ricardo Estrada}
\address{R. Estrada, Department of Mathematics\\
Louisiana State University\\
Baton Rouge, LA 70803\\
U.S.A.}
\email{restrada@math.lsu.edu}
\author{Jasson Vindas}
\address{J. Vindas, Department of Mathematics\\
Ghent University\\
Krijgslaan 281 Gebouw S22\\
B 9000 Gent\\
Belgium}
\email{jvindas@cage.Ugent.be }
\thanks{R. Estrada gratefully acknowledges support from NSF, through grant number 0968448.}

\begin{abstract}
It is well known, as follows from the Banach-Steinhaus theorem, that if a
sequence $\left\{  y_{n}\right\}  _{n=1}^{\infty}$ of linear continuous
functionals in a Fr\'{e}chet space converges pointwise to a linear functional
$Y,$ $Y\left(  x\right)  =\lim_{n\rightarrow\infty}\left\langle y_{n}%
,x\right\rangle $ for all $x,$ then $Y$ is actually continuous. In this
article we prove that in a Fr\'{e}chet space\ the continuity of $Y$ still
holds if $Y$ is the \emph{finite part} of the limit of $\left\langle
y_{n},x\right\rangle $ as $n\rightarrow\infty.$ We also show that the
continuity of finite part limits holds for other classes of topological vector
spaces, such as \textsl{LF}-spaces, \textsl{DFS}-spaces, and \textsl{DFS}%
$^{\ast}$-spaces,\ and give examples where it does not hold.

\end{abstract}
\maketitle

\section{Introduction}

Let $X$ be a topological vector space over $K$, $K$ being $\mathbb{R}$ or
$\mathbb{C}.$ We denote as $X^{\prime}$ the dual space, that is, the space of
continuous linear functionals on $X;$ the evaluation of $y\in X^{\prime}$ on
$x\in X$ will be denoted as $\left\langle y,x\right\rangle $ or as $y\left(
x\right)  ;$ we shall denote as $X_{\mathrm{al}}^{\prime}$ the algebraic dual
of $X,$ but if $z\in X_{\mathrm{al}}^{\prime}$ we denote evalutions as
$z\left(  x\right)  $ only.

Let $\left\{  y_{n}\right\}  _{n=1}^{\infty}$ be a sequence of elements of
$X^{\prime}$ and suppose that%
\begin{equation}
\lim_{n\rightarrow\infty}\left\langle y_{n},x\right\rangle =Y\left(  x\right)
\,, \label{I.1}%
\end{equation}
exists for each $x\in X,$ thus defining a function $Y:X\rightarrow K.$ It is
clear that $Y$ is linear, an element of $X_{\mathrm{al}}^{\prime},$ and simple
examples show that $Y$ does not have to be continuous, that is, maybe $Y\notin
X^{\prime}.$ However, it is well known \cite{Horvath, Treves}\ that if $X$ is
barreled, in particular if $X$ is a Fr\'{e}chet space or an \textsl{LF} space,
then one must have that $Y\in X^{\prime};$ this result is quite important in
the theory of distributions since the usual spaces of test functions are
barreled and thus (\ref{I.1}) provides a method, rather frequently employed,
to construct new distributions as limits.

Our aim is to consider the continuity of $Y$ in case the standard\footnote{We
shall consider more complicated finite part limits later on.} \emph{finite
part} of the limit%
\begin{equation}
\mathrm{F.p.}\lim_{n\rightarrow\infty}\left\langle y_{n},x\right\rangle
=Y\left(  x\right)  \,, \label{I.2}%
\end{equation}
exists for each $x\in X.$ We will show that $Y\in X^{\prime}$ in case $X$ is a
Fr\'{e}chet space or in case it is an inductive limit of Fr\'{e}chet spaces.
Naturally several distributions are defined as finite parts, so such a result
would be very useful.

The meaning of (\ref{I.2}) is that for each $x\in X$ there is $k=k_{x}
\in\mathbb{N}_{0}=\mathbb{N}\cup\{0\},$ exponents $\alpha_{1}>\cdots>\alpha_{k}>0,$ scalars
$R_{\alpha_{1}}\left(  x\right)  ,\ldots,R_{\alpha_{k}}\left(  x\right)  \in
K\setminus\left\{  0\right\}  ,$ and $z_{n}\left(  x\right)  \in K$ for
$n\geq0$ such that%
\begin{equation}
\left\langle y_{n},x\right\rangle =n^{\alpha_{1}}R_{\alpha_{1}}\left(
x\right)  +\cdots+n^{\alpha_{k}}R_{\alpha_{k}}\left(  x\right)  +z_{n}\left(
x\right)  \,, \label{I.3}%
\end{equation}
where%
\begin{equation}
\lim_{n\rightarrow\infty}z_{n}\left(  x\right)  =Y\left(  x\right)  \,.
\label{I.4}%
\end{equation}
Observe that $k_{x}$ could be $0,$ meaning that $\left\langle y_{n}%
,x\right\rangle =z_{n}\left(  x\right)  $ converges to $Y\left(  x\right)  . $

We call $n^{\alpha_{1}}R_{\alpha_{1}}\left(  x\right)  +\cdots+n^{\alpha_{k}%
}R_{\alpha_{k}}\left(  x\right)  $ the \emph{infinite} part of $\left\langle
y_{n},x\right\rangle $ as $n\rightarrow\infty$ and $z_{n}\left(  x\right)  $
the \emph{finite} part. Clearly the infinite and finite part, if they exist,
are \emph{uniquely determined, }so that the finite part of the limit, if it
exists, is likewise uniquely determined.

It is important to point out that maybe $\sup\left\{  k_{x}:x\in X\right\}
=\infty$ and that the set of exponents,%
\begin{equation}
\bigcup_{x\in X}\left\{  \alpha_{j}:1\leq j\leq k_{x}\right\}  \,, \label{I.5}%
\end{equation}
does \emph{not} have to be finite. We shall show that when $X$ is a
Fr\'{e}chet space then $\sup\left\{  k_{x}:x\in X\right\}  <\infty$ and
actually the set (\ref{I.5})\ is finite, but give examples in other types of
spaces where these results do not hold.

We shall also show that the $R_{\alpha}$ admit extensions as elements of the
algebraic dual $X_{\mathrm{al}}^{\prime},$ and show that while in general they
are not continuous, they must belong to $X^{\prime}$ when $X$ is a Fr\'{e}chet
space or an inductive limit of Fr\'{e}chet spaces.

The plan of the article is as follows. In Section \ref{Sect: General results}
we give some basic facts about finite parts that hold in any topological
vector space. The central part of the article is Section
\ref{Section: Finite parts in a Frechet space}, where we study finite parts in
a Fr\'{e}chet space. Extensions to more general finite parts and to more
general topological vector spaces are considered in Sections
\ref{Section: More General Finite Parts}\ and
\ref{Section: Other types of topological vector spaces}, respectively. Finally
we present several illustrations in Section \ref{Sect: Examples}.

\section{General results\label{Sect: General results}}

We shall first consider several results that hold in any topological vector space.

Thus let $\left\{  y_{n}\right\}  _{n=1}^{\infty}$ be a sequence of elements
of the dual space $X^{\prime}$ of the topological vector space $X,$ and
suppose that for each $x\in X$ the finite part of the limit
\begin{equation}
Y\left(  x\right)  =\mathrm{F.p.}\lim_{n\rightarrow\infty}\left\langle
y_{n},x\right\rangle \,, \label{G.0}%
\end{equation}
exists, or, in other words, that the evaluation $\left\langle y_{n}%
,x\right\rangle $ can be decomposed as
\begin{equation}
\left\langle y_{n},x\right\rangle =u_{n}\left(  x\right)  +z_{n}\left(
x\right)  \,, \label{G.1}
\end{equation}
with the infinite part of the form
\begin{equation}
u_{n}\left(  x\right)  =n^{\alpha_{1}}R_{\alpha_{1}}\left(  x\right)
+\cdots+n^{\alpha_{k}}R_{\alpha_{k}}\left(  x\right)  \,, \label{G.2}%
\end{equation}
where $\alpha_{1}>\cdots>\alpha_{k}>0,$ and $R_{\alpha_{1}}\left(  x\right)
,\ldots,R_{\alpha_{k}}\left(  x\right)  \in K\setminus\left\{  0\right\}  ,$
and with the finite part, $z_{n}\left(  x\right)  ,$ such that the limit%
\begin{equation}
Y\left(  x\right)  =\lim_{n\rightarrow\infty}z_{n}\left(  x\right)
\label{G.3}%
\end{equation}
exists.

The following result is very easy to prove, but it is also very
important.\smallskip

\begin{lemma}
\label{Lemma 1}The decomposition (\ref{G.1}) in finite and infinite parts is
unique.\smallskip
\end{lemma}

It is convenient to define $R_{\alpha}\left(  x\right)  $ for all $\alpha>0$
and all $x\in X.$ We just put $R_{\alpha}\left(  x\right)  =0$ if $\alpha$ is
not one of the exponents $\alpha_{1},\ldots,\alpha_{k}$ in the expression of
the infinite part of $\left\langle y_{n},x\right\rangle .$ This allows us to
rewrite (\ref{G.2}) as%
\begin{equation}
u_{n}\left(  x\right)  =\sum_{\alpha>0}n^{\alpha}R_{\alpha}\left(  x\right)
\,, \label{G.4}%
\end{equation}
since only a finite number of terms of the uncountable sum do not
vanish.\smallskip

\begin{lemma}
\label{Lemma 2}If $Y\left(  x\right)  =\mathrm{F.p.}\lim_{n\rightarrow\infty
}\left\langle y_{n},x\right\rangle $ exists for all $x\in X$ then $Y$ is
linear: $Y\in X_{\mathrm{al}}^{\prime}.$
\end{lemma}

\begin{proof}
Indeed, if $x_{1},x_{2}\in X$ and $c\in K,$ then $\left\langle y_{n}%
,x_{1}+cx_{2}\right\rangle $ admits the decomposition%
\begin{align*}
\left\langle y_{n},x_{1}+cx_{2}\right\rangle  &  =\left\langle y_{n}%
,x_{1}\right\rangle +c\left\langle y_{n},x_{2}\right\rangle \\
&  =\left(  u_{n}\left(  x_{1}\right)  +cu_{n}\left(  x_{2}\right)  \right)
+\left(  z_{n}\left(  x_{1}\right)  +cz_{n}\left(  x_{2}\right)  \right)  \,.
\end{align*}
Since%
\begin{equation}
u_{n}\left(  x_{1}\right)  +cu_{n}\left(  x_{2}\right)  =\sum_{\alpha
>0}n^{\alpha}\left(  R_{\alpha}\left(  x_{1}\right)  +cR_{\alpha}\left(
x_{2}\right)  \right)  \,, \label{G.5}%
\end{equation}
has the form of an infinite part, the Lemma \ref{Lemma 1} yields%
\[
u_{n}\left(  x_{1}+cx_{2}\right)  =u_{n}\left(  x_{1}\right)  +cu_{n}\left(
x_{2}\right)  \,,\ \ \ \ \ \ z_{n}\left(  x_{1}+cx_{2}\right)  =z_{n}\left(
x_{1}\right)  +cz_{n}\left(  x_{2}\right)  \,,
\]
and consequently $Y\left(  x_{1}+cx_{2}\right)  $ equals%
\[
\lim_{n\rightarrow\infty}z_{n}\left(  x_{1}+cx_{2}\right)  =\lim
_{n\rightarrow\infty}\left(  z_{n}\left(  x_{1}\right)  +cz_{n}\left(
x_{2}\right)  \right)  \,,
\]
that is, $Y\left(  x_{1}\right)  +cY\left(  x_{2}\right)  .$\smallskip
\end{proof}

If we now use the fact that $R_{\alpha}\left(  x\right)  =\mathrm{F.p.}%
\lim_{n\rightarrow\infty}n^{-\alpha}\left\langle y_{n},x\right\rangle ,$ or
employ (\ref{G.5}), we immediately obtain the ensuing result.\smallskip

\begin{lemma}
\label{Lemma 3}For any $\alpha>0$ the function $R_{\alpha}$ is linear,
$R_{\alpha}\in X_{\mathrm{al}}^{\prime}.$\smallskip
\end{lemma}

In general $Y$ nor all $R_{\alpha}$\ will not be continuous, as the Example
\ref{Example 4} shows.\smallskip

\section{Finite parts in a Fr\'{e}chet
space\label{Section: Finite parts in a Frechet space}}

We shall now consider the continuity and structure of finite parts in a
Fr\'{e}chet space. We start with some useful preliminary results.\smallskip

\begin{lemma}
\label{Lemma Fr.1}Let $X$ be a Fr\'{e}chet space and let $\left\{
y_{n}\right\}  _{n=0}^{\infty}$ be a sequence of non zero elements of
$X^{\prime}.$ Then there exists $x\in X$ such that $\left\langle
y_{n},x\right\rangle \neq0$ $\forall n\in\mathbb{N}_{0}.$
\end{lemma}

\begin{proof}
Indeed, since $y_{n}\neq0$ the kernel of $y_{n},$ $F_{n}=\left\{  x\in
X:\left\langle y_{n},x\right\rangle =0\right\}  $ is a closed proper subspace
of $X$ and thus of first category. Hence $\cup_{n=0}^{\infty}F_{n}\neq
X.$\smallskip
\end{proof}

Observe that this result fails in spaces that are not Fr\'{e}chet. Consider,
for example, the sequence $\left\{  \delta\left(  t-n\right)  \right\}
_{n=0}^{\infty}$ in the space $\mathcal{D}^{\prime}\left(  \mathbb{R}\right)
.$

Recall that a function $f:W\rightarrow V,$ where $W$ and $V$ are topological
spaces, is called a Baire function of the first class if there exists a
sequence of continuous functions from $W$ to $V,$ $\left\{  f_{n}\right\}
_{n=0}^{\infty},$ such that $f\left(  w\right)  =\lim_{n\rightarrow\infty
}f_{n}\left(  w\right)  $ for \emph{all} elements $w\in W.$\smallskip

\begin{lemma}
\label{Lemma Fr.2}Let $X$ be a Fr\'{e}chet space and let $A:X\rightarrow
\mathbb{R}$ be a function that satisfies the following three properties:

1. $A$ is a Baire function of the first class;

2. $A\left(  x-y\right)  \leq\max\left\{  A\left(  x\right)  ,A\left(
y\right)  \right\}  ;$

3. $A\left(  cx\right)  =A\left(  x\right)  $ if $c\neq0.$

Then $A$ is bounded above in $X$ and it actually attains its maximum.
\end{lemma}

\begin{proof}
We shall first show that $A$ is bounded above. If $F$ is a subset of $X,$
denote by $M_{F}=\sup\left\{  A\left(  x\right)  :x\in F\right\}  .$ If $U$ is
a neighborhood of $0,$ then \textsl{3} yields that $M_{X}=M_{U}.$ Let now $V$
be any set with non empty interior; then $V-V$ is a neighborhood of $0$ and
thus \textsl{2} yields that $M_{X}=M_{V-V}\leq M_{V}\leq M_{X}$ so that
$M_{X}=M_{V}.$

Let $\alpha_{n}$ be a sequence of continuous functions from $X$ to
$\mathbb{R}$ that converges to $A$ everywhere. Then $X=\bigcup_{k=0}^{\infty
}\left\{  x\in X:\alpha_{n}\left(  x\right)  \leq k\text{ },\forall n\right\}
$ so that there exists $k\in\mathbb{N}$ such that the set $V_{k}=\left\{  x\in
X:\alpha_{n}\left(  x\right)  \leq k\text{ }\forall n\right\}  $ has non empty
interior. This yields that $M_{X}=M_{V_{k}}\leq k,$ so that $A$ is bounded
above by $k$ in the \emph{whole} space $X.$

We should now show that there exists $\widetilde{x}\in X$ such that $A\left(
\widetilde{x}\right)  =M_{X}.$ If not, the function $B\left(  x\right)
=1/\left(  M_{X}-A\left(  x\right)  \right)  $ satisfies the same three
conditions as $A,$ and from what we have already proved, $B$ must be bounded
above by some constant $\lambda>0;$ but this means that $A\left(  x\right)
\leq M_{X}-1/\lambda,$ for all $x\in X,$ and consequently $M_{X}\leq
M_{X}-1/\lambda,$ a contradiction.\smallskip
\end{proof}

We now apply the Lemma \ref{Lemma Fr.2} to the study of finite parts. Indeed,
let $\left\{  y_{n}\right\}  _{n=1}^{\infty}$ be a sequence of elements of the
dual space $X^{\prime}$ of the Fr\'{e}chet space $X,$ and suppose that for
each $x\in X$ the finite part of the limit $Y\left(  x\right)  =\mathrm{F.p.}%
\lim_{n\rightarrow\infty}\left\langle y_{n},x\right\rangle $ exists. If the
infinite part of $\left\langle y_{n},x\right\rangle $ has the expression as a
finite sum,
\begin{equation}
u_{n}\left(  x\right)  =n^{\alpha_{1}}R_{\alpha_{1}}\left(  x\right)
+\cdots+n^{\alpha_{k}}R_{\alpha_{k}}\left(  x\right)  =\sum_{\alpha
>0}n^{\alpha}R_{\alpha}\left(  x\right)  \,, \label{Fr.0}%
\end{equation}
where $\alpha_{1}>\cdots>\alpha_{k}>0,$ and $R_{\alpha_{1}}\left(  x\right)
,\ldots,R_{\alpha_{k}}\left(  x\right)  \in K\setminus\left\{  0\right\}  ,$
define $A\left(  x\right)  =0$ if $u_{n}\left(  x\right)  =0$ and as%
\begin{equation}
A\left(  x\right)  =\alpha_{1}=\max\left\{  \alpha>0:R_{\alpha}\left(
x\right)  \neq0\right\}  \,, \label{Fr.1}%
\end{equation}
otherwise.\smallskip

\begin{lemma}
\label{Lemma Fr.3}The function $A$ is bounded above and attains its maximum in
$X.$
\end{lemma}

\begin{proof}
It is enough to prove that $A$ satisfies the three conditions of the Lemma
\ref{Lemma Fr.2}. However, condition \textsl{1} follows from the limit formula%
\begin{equation}
A\left(  x\right)  =\lim_{n\rightarrow\infty}\frac{\ln\left[  \left\vert
\left\langle y_{n},x\right\rangle \right\vert +1\right]  }{\ln n}\,,
\label{Fr.2}%
\end{equation}
while \textsl{2} and \textsl{3} are obvious.\smallskip
\end{proof}

The Lemma \ref{Lemma Fr.3} not only means that if $\widetilde{\alpha}%
=\max\left\{  A\left(  x\right)  :x\in X\right\}  $ then $R_{\alpha}\left(
x\right)  =0$ if $\alpha>\widetilde{\alpha},$ but it also means that
$R_{\widetilde{\alpha}}\neq0.$ The linear form $R_{\widetilde{\alpha}}$ is
actually \emph{continuous} as follows from the Banach-Steinhaus theorem since%
\begin{equation}
R_{\widetilde{\alpha}}\left(  x\right)  =\lim_{n\rightarrow\infty
}n^{-\widetilde{\alpha}_{1}}\left\langle y_{n},x\right\rangle \,, \label{Fr.3}%
\end{equation}
for each $x\in X.$

We can then replace $y_{n}$ by $y_{n}-n^{\widetilde{\alpha}}R_{\widetilde
{\alpha}}$ and apply the same ideas as above. Therefore, for \emph{some}
integers $k$ we obtain exponents $\widetilde{\alpha}_{1}=\widetilde{\alpha
}>\cdots>\widetilde{\alpha}_{k}>0$ such that $R_{\alpha}\left(  x\right)  =0$
if $\alpha>\widetilde{\alpha}_{k},$ $\alpha\neq\widetilde{\alpha}_{j},$ $1\leq
j<k,$ while $R_{\widetilde{\alpha}_{j}}$ is continuous and $R_{\widetilde
{\alpha}_{j}}\neq0$ for $1\leq j\leq k.$ In principle one could think that
this is possible for each $k\geq0,$ but if it were then we would obtain an
infinite sequence of non zero continuous functionals $\left\{  R_{\widetilde
{\alpha}_{j}}\right\}  _{j=0}^{\infty}$ and the Lemma \ref{Lemma Fr.1} would
give us the existence of $x^{\ast}\in X$ such that $R_{\widetilde{\alpha}_{j}%
}\left(  x^{\ast}\right)  \neq0$ for all $j,$ a contradiction, since for any
$x\in X$ the set $\left\{  \alpha>0:R_{\alpha}\left(  x\right)  \neq0\right\}
$ is finite. Summarizing, we have the following result.\smallskip

\begin{theorem}
\label{Theorem Fr.1}Let $\left\{  y_{n}\right\}  _{n=1}^{\infty}$ be a
sequence of elements of the dual space $X^{\prime}$ of the Fr\'{e}chet space
$X,$ and suppose that for each $x\in X$ the finite part of the limit $Y\left(
x\right)  =\mathrm{F.p.}\lim_{n\rightarrow\infty}\left\langle y_{n}%
,x\right\rangle $ exists. Then there exists $k\in\mathbb{N},$ exponents
$\widetilde{\alpha}_{1}>\cdots>\widetilde{\alpha}_{k}>0,$ and continuous non
zero linear functionals $\left\{  R_{\widetilde{\alpha}_{j}}\right\}
_{j=1}^{k}$ such that for all $x\in X$%
\begin{equation}
\left\langle y_{n},x\right\rangle =n^{\widetilde{\alpha}_{1}}R_{\widetilde
{\alpha}_{1}}\left(  x\right)  +\cdots+n^{\widetilde{\alpha}_{k}}%
R_{\widetilde{\alpha}_{k}}\left(  x\right)  +z_{n}\left(  x\right)
\,,\label{Fr.4}%
\end{equation}
where the finite part $z_{n}$ is continuous for all $n$ and where
\begin{equation}
Y\left(  x\right)  =\lim_{n\rightarrow\infty}z_{n}\left(  x\right)
\,,\label{Fr.5}%
\end{equation}
is also a continuous linear functional on $X.$
\end{theorem}

\begin{proof}
The only thing left to prove is the continuity of the $z_{n}$'s and the
continuity of $Y.$ But the continuity of the $R_{\widetilde{\alpha}_{j}}$'s
yields the continuity of the $z_{n}$'s because of (\ref{Fr.4}) while the
continuity of $Y$ follows from (\ref{Fr.5}) and the Banach-Steinhaus theorem.
\end{proof}

\section{More general finite parts\label{Section: More General Finite Parts}}

One can consider a general finite part limit process as follows. Let
$\Lambda\cup\left\{  \lambda_{0}\right\}  $ be a topological space where
$\lambda_{0}\in\overline{\Lambda}\setminus\Lambda,$ and let $\left(
\mathfrak{E,}\prec\right)  $ be a totally ordered set. Let $\mathsf{B}%
=\left\{  \rho_{\alpha}\right\}  _{\alpha\in\mathfrak{E}}$ be the
\textquotedblleft basic infinite functions,\textquotedblright\ that is, a
family of functions with the following properties:\smallskip

\begin{enumerate}
\item For each $\alpha\in\mathfrak{E,}$ $\rho_{\alpha}:\Lambda\rightarrow
\left(  0,\infty\right)  ,$ and $\lim_{\lambda\rightarrow\lambda_{0}}%
\rho_{\alpha}\left(  \lambda\right)  =\infty;$\smallskip

\item If $\alpha\prec\beta$ then $\rho_{\alpha}\left(  \lambda\right)
=o\left(  \rho_{\beta}\left(  \lambda\right)  \right)  $ as $\lambda
\rightarrow\lambda_{0}.$\smallskip
\end{enumerate}

Let $y_{\lambda}\in K,$ where $K$ is $\mathbb{R}$ or $\mathbb{C},$\ for each
$\lambda\in\Lambda.$ If we can write%
\begin{equation}
y_{\lambda}=u_{\lambda}+z_{\lambda}\,,\label{MGG.1}%
\end{equation}
where the \textquotedblleft infinite part\textquotedblright\ has the form%
\begin{equation}
u_{\lambda}=\sum_{j=1}^{k}R_{\alpha_{j}}\rho_{\alpha_{j}}\left(
\lambda\right)  \,,\label{MGG.2}%
\end{equation}
where $\alpha_{1}\succ\cdots\succ\alpha_{k},$ and $R_{\alpha_{1}}%
,\ldots,R_{\alpha_{k}}\in K\setminus\left\{  0\right\}  ,$ and where the
\textquotedblleft finite part,\textquotedblright\ $z_{\lambda},$ satisfies
that the limit%
\begin{equation}
Y=\lim_{\lambda\rightarrow\lambda_{0}}z_{\lambda}\label{MGG.3}%
\end{equation}
exists, then we say that the finite part of the limit of $y_{\lambda}$ as
$\lambda\rightarrow\lambda_{0}$ with respect to $\mathsf{B}$ exists and equals
$Y,$ and write\footnote{Let $V$ be the vector space of all functions of the
form $\sum_{j=1}^{k}c_{j}\rho_{\alpha_{j}}+\mu$ with $\lim_{\lambda
\rightarrow\lambda_{0}}\mu(\lambda)=0$. The triple $N=(\Lambda,K,V)$ forms a
\emph{neutrix} in the sense of van der Corput \cite{Corput}. In his
terminology, the finite part limit (\ref{MGG.4}) coincides with the neutrix
value $y_{N}$.}
\begin{equation}
Y=\mathrm{F.p.}_{\mathsf{B}}\lim_{\lambda\rightarrow\lambda_{0}}y_{\lambda
}\,.\label{MGG.4}%
\end{equation}

We have considered the \emph{standard} system $\mathsf{B}=\left\{
\rho_{\alpha}\right\}  _{\alpha>0}$ where $\Lambda=\mathbb{N},$ $\lambda
_{0}=\infty,$ and $\rho_{\alpha}\left(  \lambda\right)  =\lambda^{\alpha}.$
Naturally one can consider the same standard system for functions defined in
any unbounded set $\Lambda\subset(0,\infty),$ in particular for $\Lambda
=(0,\infty).$

We can also consider \emph{Hadamard finite part} limits\footnote{Hadamard was
probably the first to use finite parts; in his 1923 work \cite{Hadamard}, he
employs them to find fundamental solutions of partial differential
equations.}, where the infinite basic functions are products of powers and
powers of logarithms. Explicitly, let $\mathfrak{E}=[0,\infty)^{2}%
\setminus\left\{  \left(  0,0\right)  \right\}  ,$ with the order given by
$\left(  \alpha_{1},\beta_{1}\right)  \prec\left(  \alpha_{2},\beta
_{2}\right)  $ if $\alpha_{1}<\alpha_{2}$ or if $\alpha_{1}=\alpha_{2}$ and
$\beta_{1}<\beta_{2}.$ Here $\Lambda\subset(1,\infty),$ and the basic infinite
functions $\mathsf{B}=\left\{  \rho_{\left(  \alpha,\beta\right)  }\right\}
_{\left(  \alpha,\beta\right)  \in\mathfrak{E}}$ are given as $\rho_{\left(
\alpha,\beta\right)  }\left(  \lambda\right)  =\lambda^{\alpha}\ln^{\beta
}\lambda.$

We can also take a set $\epsilon\subset\left(  0,1\right)  $ and consider
limits as $\varepsilon\in\epsilon$ tends to $0.$ For standard finite part
limits $\mathfrak{E}=(0,\infty)$ and the basic infinite functions are
$\rho_{\alpha}\left(  \varepsilon\right)  =\varepsilon^{-\alpha};$ for
Hadamard finite part limits, $\mathfrak{E}=[0,\infty)^{2}\setminus\left\{
\left(  0,0\right)  \right\}  $ and $\rho_{\left(  \alpha,\beta\right)
}\left(  \varepsilon\right)  =\varepsilon^{-\alpha}\left\vert \ln
\varepsilon\right\vert ^{\beta}.$

The continuity of the finite part of the limit in Fr\'{e}chet spaces, Theorem
\ref{Theorem Fr.1}, will also hold for these more general systems of basic
infinite functions.\ This is of course the case for standard finite limits.
For Hadamard finite parts the proof can be modified as follows. Indeed, let
$X$ be a Fr\'{e}chet space, $\Lambda\subset(1,\infty)$ is an unbounded set and
$y_{\lambda}\in X^{\prime}$ for each $\lambda\in\Lambda.$ Suppose that for
each $x\in X$ the evaluation $\left\langle y_{\lambda},x\right\rangle $ can be
written as
\begin{equation}
\left\langle y_{\lambda},x\right\rangle =u_{\lambda}\left(  x\right)
+z_{\lambda}\left(  x\right)  \,,\label{MG.1}%
\end{equation}
where the infinite part has the ensuing form for some $k=k_{x},$%
\begin{equation}
u_{\lambda}\left(  x\right)  =\lambda^{\alpha_{1}}\ln^{\beta_{1}}%
\lambda\,R_{\left(  \alpha_{1},\beta_{1}\right)  }\left(  x\right)
+\cdots+\lambda^{\alpha_{k}}\ln^{\beta_{k}}\lambda\,R_{\left(  \alpha
_{k},\beta_{k}\right)  }\left(  x\right)  \,,\label{MG.2}%
\end{equation}
where the exponents $\left(  \alpha_{j},\beta_{j}\right)  \in\mathfrak{E}$
satisfy$\ \left(  \alpha_{1},\beta_{1}\right)  \succ\cdots\succ\left(
\alpha_{k},\beta_{k}\right)  ,$ where $R_{\left(  \alpha_{1},\beta_{1}\right)
}\left(  x\right)  ,\ldots,R_{\left(  \alpha_{k},\beta_{k}\right)  }\left(
x\right)  \in K\setminus\left\{  0\right\}  ,$ and where the finite part,
$z_{\lambda}\left(  x\right)  ,$ satisfies that the limit%
\begin{equation}
Y\left(  x\right)  =\lim_{\lambda\rightarrow\infty,\lambda\in\Lambda
}z_{\lambda}\left(  x\right)  \,,\label{MG.3}%
\end{equation}
exists. As before, we set $R_{\left(  \alpha,\beta\right)  }\left(  x\right)
=0$ if $\left(  \alpha,\beta\right)  \neq\left(  \alpha_{j},\beta_{j}\right)
$ for $1\leq j\leq k_{x}.$ Then we have the following generalization of the
Lemma \ref{Lemma Fr.3}.\smallskip

\begin{lemma}
\label{Lemma GM.1}Let%
\begin{equation}
E\left(  x\right)  =\left(  A\left(  x\right)  ,B\left(  x\right)  \right)
=\max\left\{  \left(  \alpha,\beta\right)  \in\mathfrak{E}:R_{\left(
\alpha,\beta\right)  }\left(  x\right)  \neq0\right\}  \,. \label{MG.5}%
\end{equation}
Then $E$ attains its maximum, $\left(  \alpha^{\ast},\beta^{\ast}\right)  ,$
and $R_{\left(  \alpha^{\ast},\beta^{\ast}\right)  }$ is continuous and not zero.
\end{lemma}

\begin{proof}
The proof of the Lemma \ref{Lemma Fr.3} applies to $A,$ so there exists
$\alpha^{\ast}=\max_{x\in X}A\left(  x\right)  .$ For this exponent
$\alpha^{\ast}$ we consider the function given by $B^{\ast}\left(  x\right)
=0$ if $R_{\left(  \alpha^{\ast},\beta\right)  }\left(  x\right)  =0$ for all
$\beta$ and otherwise by%
\begin{equation}
B^{\ast}\left(  x\right)  =\max\left\{  \beta:R_{\left(  \alpha^{\ast}%
,\beta\right)  }\left(  x\right)  \neq0\right\}  \,. \label{MG.6}%
\end{equation}
The Lemma \ref{Lemma Fr.2} yields the existence of $\beta^{\ast}=\max_{x\in
X}B^{\ast}\left(  x\right)  $ because%
\begin{equation}
B^{\ast}\left(  x\right)  =\lim_{\lambda\rightarrow\infty}\frac{\ln\left[
\lambda^{-\alpha^{\ast}}\left\vert \left\langle y_{\lambda},x\right\rangle
\right\vert +1\right]  }{\ln\ln\lambda}\,. \label{MG.7}%
\end{equation}
Since%
\begin{equation}
R_{\left(  \alpha^{\ast},\beta^{\ast}\right)  }\left(  x\right)
=\lim_{\lambda\rightarrow\infty}(\lambda^{-\alpha^{\ast}}\ln^{-\beta^{\ast}%
}\lambda)\left\langle y_{\lambda},x\right\rangle \,, \label{MG.8}%
\end{equation}
the continuity of $R_{\left(  \alpha^{\ast},\beta^{\ast}\right)  }$
follows.\smallskip
\end{proof}

Therefore we obtain that \emph{the Theorem \ref{Theorem Fr.1} holds for
Hadamard finite parts.}

These ideas can be further generalized. Let us take $m$ positive functions
$F_{1},F_{2},F_{3},\dots,F_{m}$ defined on an unbounded set $\Lambda
\subset(1,\infty),$ each of them tending to $\infty$ as $\lambda
\rightarrow\infty$ and such that $F_{j+1}(\lambda)=o(F_{j}^{\alpha}(\lambda))$
as $\lambda\rightarrow\infty$ for all $\alpha>0.$ We now consider
$\mathfrak{E}=[0,\infty)^{m}\setminus\left\{  (0,0,\dots,0)\right\}  $ with
the lexicographical order $\prec.$ Set $\mathbf{F}=(F_{1},F_{2},\dots,F_{m})$
and if $\vec{\alpha}\in\mathfrak{E}$ write $\mathbf{F}^{\vec{\alpha}}%
=F_{1}^{\alpha_{1}}F_{2}^{\alpha_{2}}\cdots F_{m}^{\alpha_{m}}$, where
$\vec{\alpha}=(\alpha_{1},\alpha_{2},\dots,\alpha_{m})$. We now choose the
basic infinite functions as $\mathsf{B}=\left\{  \mathbf{F}^{\vec{\alpha}%
}\right\}  _{\vec{\alpha}\in\mathfrak{E}}$\ . If $y_{\lambda}\in X^{\prime}$
for each $\lambda\in\Lambda$, where $X$ is again a Fr\'{e}chet space, and for
each $x\in X$ the evaluation $\left\langle y_{\lambda},x\right\rangle $ can be
decomposed as in (\ref{MG.1}) where the infinite part is now taken of the form
(for some $k=k_{x},$)
\begin{equation}
u_{\lambda}\left(  x\right)  =\mathbf{F}^{\vec{\alpha}_{1}}(\lambda
)R_{\vec{\alpha}_{1}}\left(  x\right)  +\cdots+\mathbf{F}^{\vec{\alpha}_{k}%
}(\lambda)R_{\vec{\alpha}_{1}}\left(  x\right)  \,,\label{MG.12}%
\end{equation}
with $\vec{\alpha}_{1}\succ\cdots\succ\vec{\alpha}_{k},$ and $R_{\vec{\alpha
}_{1}}\left(  x\right)  ,\ldots,R_{\vec{\alpha}_{k}}\left(  x\right)  \in
K\setminus\left\{  0\right\}  ,$ and where $z_{\lambda}\left(  x\right)  $
satisfies (\ref{MG.3}), we can then define the finite part limit of
$y_{\lambda}(x)$ as $Y(x)$. Defining $R_{\vec{\alpha}}(x)=0$ if $\vec{\alpha}$
does not occur in (\ref{MG.12}), the proof of Lemma \ref{Lemma GM.1} can be
readily adapted to show that the function $\mathbf{A}(x)=\max\left\{
\vec{\alpha}\in\mathfrak{E}:R_{\vec{\alpha}}\left(  x\right)  \neq0\right\}  $
also attains its maximum, $\vec{\alpha}^{\ast},$ and that $R_{\vec{\alpha
}^{\ast}}\in X^{\prime}\setminus\{0\}$. This leads to a general version of
Theorem \ref{Theorem Fr.1} for finite part limits with respect to the system
of infinite functions $\mathsf{B}=\left\{  \mathbf{F}^{\vec{\alpha}}\right\}
_{\vec{\alpha}\in\mathfrak{E}}$\ . Naturally, the Hadamard finite part
corresponds to the choices $m=2$, $F_{1}(\lambda)=\lambda$, and $F_{2}%
(\lambda)=\log\lambda$.

\section{Other types of topological vector
spaces\label{Section: Other types of topological vector spaces}}

The continuity of finite part limits holds not only in Fr\'{e}chet spaces, but
in other types of spaces, those that carry a final locally convex topology
given by a family of Fr\'{e}chet spaces \cite{Horvath}. Indeed, let
$(X_{i},u_{i})_{I}$ be a family of Fr\'{e}chet spaces and linear mappings
$u_{i}:X_{i}\rightarrow X$ for each $i\in I$. If $X$ is provided with the
finest locally convex topology that makes all mappings $u_{i}$ continuous, the
continuity of the finite part limits and the $R_{\vec{\alpha}}$'s follows at
once from the fact that $y\in X^{\prime}$ if and only if $y\circ u_{i}\in
X_{i}^{\prime}$, $\forall i\in I$. In particular, the result holds for
\emph{any} inductive limit of an inductive system of Fr\'{e}chet spaces.
Important instances of such inductive limits are those that can be written as
countable inductive unions of Fr\'{e}chet spaces, such as the \textsl{LF}%
-spaces, the \textsl{DFS}-spaces, and the \textsl{DFS}$^{\ast}$-spaces
\cite{komatsu}.

In these more general spaces, however, the set of exponents for which
$R_{\vec{\alpha}}\left(  x\right)  \neq0$ does not have to be finite, not
bounded, in general (see Examples \ref{Example 2}, \ref{Example 3}, and
\ref{Example 3.5} below).\smallskip

\begin{theorem}
\label{Theorem other 1}Let $X$ be a locally convex space that is the inductive
limit of a system of Fr\'{e}chet spaces. Let $\Lambda$ be an unbounded subset
of $\left(  1,\infty\right)  $ and for each $\lambda\in\Lambda$ let
$y_{\lambda}\in X^{\prime}.$ Suppose that for each $x\in X$ the finite part of
the limit $Y\left(  x\right)  =\mathrm{F.p.}\lim_{\lambda\rightarrow\infty
}\left\langle y_{\lambda},x\right\rangle $ exists. Then $Y\in X^{\prime}.$

Likewise, $R_{\vec{\alpha}}$ is continuous for each $\vec{\alpha},$ but while
for each $x\in X$ the set $\left\{  \vec{\alpha}:R_{\vec{\alpha}}\left(
x\right)  \neq0\right\}  $ is finite, the set $\left\{  \vec{\alpha}%
:R_{\vec{\alpha}}\neq0\right\}  $ could be infinite and not bounded above.
\end{theorem}

\section{Examples\label{Sect: Examples}}

In order to better understand our results, it is useful to look at several
examples.\smallskip

\begin{example}
\label{Example 1}The best known example of finite parts are the distributions
constructed as the finite part of divergent integrals \cite{GreenBook,
Kanwal}. Suppose $G$ is a homogenous continuous function in $\mathbb{R}%
^{d}\setminus\left\{  \mathbf{0}\right\}  ,$ homogenous of degree $\lambda
\in\mathbb{R}.$ Then $G$ gives a well defined distribution of the space
$\mathcal{D}^{\prime}\left(  \mathbb{R}^{d}\setminus\left\{  \mathbf{0}%
\right\}  \right)  ,$ which without loss of generality we can still denote by
$G,$ as $\left\langle G,\phi\right\rangle =\int_{\mathbb{R}^{d}\setminus
\left\{  \mathbf{0}\right\}  }G\left(  \mathbf{x}\right)  \phi\left(
\mathbf{x}\right)  \,\mathrm{d}\mathbf{x},$ for $\phi\in\mathcal{D}\left(
\mathbb{R}^{d}\setminus\left\{  \mathbf{0}\right\}  \right)  .$ When
$\lambda<-d$ then the integral $\int_{\mathbb{R}^{d}}G\left(  \mathbf{x}%
\right)  \phi\left(  \mathbf{x}\right)  \,\mathrm{d}\mathbf{x},$ would be
divergent in general if $\phi\in\mathcal{D}\left(  \mathbb{R}^{d}\right)  $
and thus there is no canonical distribution corresponding to $G$ in
$\mathcal{D}^{\prime}\left(  \mathbb{R}^{d}\right)  .$ One can, however,
define the distribution $\mathrm{F.p.}\left(  G\right)  ,$ the \emph{radial}
finite part\footnote{If instead of removing \emph{balls} of small radius,
solids of other shapes are removed one obtains a different finite part
distribution \cite{Farassat, VH1, YE}, an important fact in the numerical
solution of integral equations \cite{Farassat}. The known formulas for the
distributional derivatives of inverse power fields \cite{Frahm} and the
corresponding\ finite parts \cite{EK85b, EK88}\ hold for \emph{radial} finite
parts.} of $G$ by setting for $\phi\in\mathcal{D}\left(  \mathbb{R}%
^{d}\right)  $%
\begin{equation}
\left\langle \mathrm{F.p.}\left(  G\right)  ,\phi\right\rangle =\mathrm{F.p.}%
\lim_{n\rightarrow\infty}\int_{\left\vert \mathbf{x}\right\vert \geq
1/n}G\left(  \mathbf{x}\right)  \phi\left(  \mathbf{x}\right)  \,\mathrm{d}%
\mathbf{x}\,, \label{Ex.1}%
\end{equation}
a standard finite part if $-\lambda\notin\mathbb{N},$ and a Hadamard finite
part if $\lambda$ is an integer. Similar ideas are needed to construct
\emph{thick} distributions from locally integrable functions \cite{YE2}%
.\smallskip
\end{example}

\begin{example}
\label{Example 2}Let $\mathbb{D}=\left\{  z\in\mathbb{C}:\left\vert
z\right\vert <1\right\}  $ be the unit disc in $\mathbb{C}.$ Let $H_{k}$ be
the Banach space of functions continuous in $\overline{\mathbb{D}}%
\setminus\left\{  0\right\}  ,$ analytic in $\mathbb{D}\setminus\left\{
0\right\}  ,$ and that have a pole at $z=0$ of order $k;$ the norm being
$\left\Vert f\right\Vert =\max_{\left\vert z\right\vert \leq1}\left\vert
z\right\vert ^{k}\left\vert f\left(  z\right)  \right\vert .$ Let $H$ be the
inductive limit of the $H_{k}$ as $k\rightarrow\infty.$ Consider the
functionals $y_{n}\in H^{\prime}$ given as%
\begin{equation}
\left\langle y_{n},f\right\rangle =f\left(  1/n\right)  \,,\label{Ex.2}%
\end{equation}
that is, $y_{n}=\delta\left(  z-1/n\right)  .$ For each $f\in H$ the finite
part of the limit $\mathrm{F.p.}\lim_{n\rightarrow\infty}\left\langle
y_{n},f\right\rangle =Y\left(  f\right)  $ exists, and equals the finite part
of $f$ at $z=0;$ in fact, if $f\left(  z\right)  =\sum_{j=1}^{k}a_{j}%
z^{-j}+g\left(  z\right)  ,$ where $g$ is analytic at $0,$ then%
\begin{equation}
\left\langle y_{n},f\right\rangle =\sum_{j=1}^{k}a_{j}n^{j}+g\left(
1/n\right)  \,.\label{Ex.3}%
\end{equation}
Observe that the infinite part is $\sum_{j=1}^{k}a_{j}n^{j},$ which has
arbitrary large exponents; here the set (\ref{I.5}) is infinite. Also
$Y\left(  f\right)  =g\left(  0\right)  ,$ the usual finite part of the
analytic function at the pole. Our results will yield the continuity of $Y,$
but one can prove this directly, for example, by observing that $Y\left(
f\right)  =\left(  2\pi i\right)  ^{-1}%
%TCIMACRO{\doint _{\left\vert z\right\vert =r}}%
%BeginExpansion
{\displaystyle\oint_{\left\vert z\right\vert =r}}
%EndExpansion
z^{-1}f\left(  z\right)  \,\mathrm{d}z$ for any $r\in(0,1].$

Interestingly, if $X$ is the space of all analytic function in $\mathbb{D}%
\setminus\left\{  0\right\}  $ with its standard topology, then $H$ is dense
in $X$ and the $y_{n}$'s and $Y$ admit continuous extensions to $X^{\prime}, $
but the extension of $Y$ is \emph{not} the finite part of the limit of the
extensions of the $y_{n}$'s.\smallskip
\end{example}

\begin{example}
\label{Example 3}Let us consider the distributions
\begin{equation}
f_{n}\left(  x\right)  =\sum_{k=1}^{\infty}\left(  n^{1/k}+1\right)  ^{k^{2}%
}\delta\left(  x-k\right)  \,, \label{Ex.4}%
\end{equation}
of the space $\mathcal{D}^{\prime}\left(  \mathbb{R}\right)  .$ If $\phi
\in\mathcal{D}\left(  \mathbb{R}\right)  $ satisfies $\operatorname*{supp}%
\phi\subset(-\infty,k+1)$ then the infinite part of $\left\langle f_{n}%
,\phi\right\rangle $ is the sum of $k^{2}-1$ terms, corresponding to the
exponents $\alpha_{j}=j/k$ for $1\leq j\leq k^{2}.$ Hence the finite part of
the limit is the Dirac comb
\begin{equation}
\mathrm{F.p.}\lim_{n\rightarrow\infty}f_{n}\left(  x\right)  =\sum
_{k=1}^{\infty}\delta\left(  x-k\right)  \,, \label{Ex.5}%
\end{equation}
while $R_{\alpha}\left(  x\right)  \neq0$ precisely when $\alpha$ is a
positive rational number; and actually if $k$ is the smallest integer for
which $\alpha=j/k$ and $j\leq k^{2},$\ then%
\begin{equation}
R_{j/k}\left(  x\right)  =\sum_{q=1}^{\infty}\binom{q^{2}k^{2}}{q\,j}%
\delta\left(  x-qk\right)  \,. \label{Ex.6}%
\end{equation}
One could represent the infinite part of $f_{n}\left(  x\right)  $ as the
infinite sum $\sum_{\alpha\in\mathbb{Q}_{+}}R_{\alpha}\left(  x\right)  ;$ the
set of exponents $\alpha$ for which $R_{\alpha}\neq0$ is the infinite
unbounded set $\mathbb{Q}_{+},$\ but upon evaluation on a test function the
sum becomes finite since $\left\langle R_{\alpha},\phi\right\rangle \neq0$ for
only a finite set of exponents.\smallskip
\end{example}

In the previous example $X=\mathcal{D}\left(  \mathbb{R}\right)  $ is an
\textsl{LF} space, and it is not hard to see that in an \textsl{LF} space\ the
set of exponents $\alpha$ for which $R_{\alpha}\neq0$ is countable at the
most. We can easily construct an example where this set of exponents is the
whole $\left(  0,\infty\right)  .$\smallskip

\begin{example}
\label{Example 3.5}Let $X$ be the space of functions $f:\left(  0,\infty
\right)  \rightarrow\mathbb{R}$ such that the set $\left\{  \alpha\in\left(
0,\infty\right)  :f\left(  \alpha\right)  \neq0\right\}  $ is finite. We give
$X$ the inductive limit topology of the system $\left(  \mathbb{R}^{F}%
,i_{F}\right)  ,$ where $F$ is a finite subset of $\left(  0,\infty\right)  ,$
$F\nearrow,$ and if $f\in\mathbb{R}^{F},$ then $i_{F}\left(  f\right)
=f_{F}\in X$ is given by $f_{F}\left(  \alpha\right)  =f\left(  \alpha\right)
$ if $\alpha\in F$ and $f_{F}\left(  \alpha\right)  =0$ if $\alpha\notin F.$
Theorem \ref{Theorem other 1} applies in $X.$

Let $y_{n}\in X^{\prime}$ be given as
\begin{equation}
y_{n}\left(  x\right)  =\sum_{\alpha\in\left(  0,\infty\right)  }n^{\alpha
}\delta\left(  x-\alpha\right)  \,, \label{Ex.3.5.1}%
\end{equation}
that is $\left\langle y_{n},f\right\rangle $ is the \emph{finite} sum
$\sum_{\alpha\in\left(  0,\infty\right)  }n^{\alpha}f\left(  \alpha\right)  .$
Then $R_{\alpha}\left(  x\right)  =\delta\left(  x-\alpha\right)  $ for all
$\alpha\in\left(  0,\infty\right)  ,$ so that the set (\ref{I.5}) is the whole
$\left(  0,\infty\right)  .$ Notice also that $\mathrm{F.p.}\lim
_{n\rightarrow\infty}y_{n}\left(  x\right)  =0.$\smallskip
\end{example}

\begin{example}
\label{Example 4}Consider the space $X$ whose elements are the continuous
functions in $\left[  0,1\right]  ,$ with the topology of pointwise
convergence on $\left[  0,1\right]  .$ If $0<\beta<1,$ let us consider the
functional $f_{n}\in X^{\prime}$ given by%
\begin{equation}
f_{n}\left(  x\right)  =\sum_{k=0}^{n-1}\delta\left(  x-\frac{k+\beta}%
{n}\right)  \,, \label{Ex.7}%
\end{equation}
that is,%
\begin{equation}
\left\langle f_{n},\phi\right\rangle =\sum_{k=0}^{n-1}\phi\left(
\frac{k+\beta}{n}\right)  \,, \label{Ex.8}%
\end{equation}
for $\phi\in X.$ Then the Euler-Maclaurin formula \cite{Estrada97} yields%
\begin{equation}
\left\langle f_{n},\phi\right\rangle =n\int_{0}^{1}\phi\left(  x\right)
\,\mathrm{d}x+B_{1}\left(  \beta\right)  \left(  \phi\left(  1\right)
-\phi\left(  0\right)  \right)  \,, \label{Ex.9}%
\end{equation}
where $B_{1}\left(  x\right)  =x-1/2$ is the Bernoulli polynomial of order
$1.$ The finite part is%
\begin{equation}
\mathrm{F.p.}\lim_{n\rightarrow\infty}f_{n}\left(  x\right)  =B_{1}\left(
\beta\right)  \left(  \delta\left(  x-1\right)  -\delta\left(  x\right)
\right)  \,, \label{Ex.10}%
\end{equation}
which is actually continuous in $X,$ but $R_{1},$ the coefficient of $n$ in
the infinite part of $\left\langle f_{n},\phi\right\rangle $ is not continuous
since it is given by $\phi\rightsquigarrow\int_{0}^{1}\phi\left(  x\right)
\,\mathrm{d}x,$ which belongs to $X_{\mathrm{al}}^{\prime}$ but not to
$X^{\prime}.$\smallskip
\end{example}

\begin{example}
\label{Example 5}Let $\Omega$\ be a complex region and let $\xi\in\Omega.$
Suppose $y_{\omega}\in X^{\prime}$ is weakly--$\ast$ analytic in $\omega
\in\Omega\setminus\left\{  \xi\right\}  ,$ that is, for each $x\in X$ the
function $\left\langle y_{\omega},x\right\rangle $ is analytic in
$\Omega\setminus\left\{  \xi\right\}  .$ Suppose also that $\left\langle
y_{\omega},x\right\rangle $ has a pole at $\omega=\xi$ for each $x.$

If $X$ is a Fr\'{e}chet space then there exists a fixed number $N$ such that
the order of the pole is $N$ at the most for all $x,$ and the finite part
\begin{equation}
y_{\xi}^{\ast}\left(  x\right)  =\mathrm{F.p.}\lim_{n\rightarrow\infty
}\left\langle y_{\xi+1/n},x\right\rangle \,, \label{Ex.11}%
\end{equation}
is an element of $X^{\prime}.$

If $X$ is an inductive limit of Fr\'{e}chet spaces, then $y_{\xi}^{\ast}$ is
still continuous, but the order of the pole of $\left\langle y_{\omega
},x\right\rangle $\ at $\xi$ does not have to be bounded, that is, maybe
$R_{k}\left(  x\right)  =\mathrm{F.p.}\lim_{n\rightarrow\infty}n^{-k}%
\left\langle y_{\xi+1/n},x\right\rangle $ does not vanish in $X$ for an
infinite number of values of $k.$\smallskip
\end{example}

.

\end{document}